\documentclass[11pt, a4paper, twoside]{article} 
\usepackage{amsmath, amssymb, latexsym, amscd, amsthm,amsfonts,amstext}
\usepackage[mathscr]{eucal}
\usepackage{cite}

  \renewenvironment{abstract}{%
	 \if@twocolumn
        \section*{\abstractname}%
      \else
		\small{\quotation \noindent\abstractname}
      \fi}
      {\if@twocolumn\else\endquotation\fi}

    \makeatletter
    \renewcommand\section{\large\@startsection{section}{1}{\z@}%
                                      {-3.5ex \@plus -1ex \@minus -.2ex}%
                                      {2.3ex \@plus.2ex}%
                                      {\bfseries\textsf}}
    \makeatother

\newtheorem{theorem}{Theorem}
\newtheorem{definition}{Definition}
\newtheorem{lemma}{Lemma}

\newtheorem{corollary}{Corollary}

\newtheorem{example}{Example}

\renewcommand{\thesection}{\arabic{section}.}

\renewcommand{\thetable}{\arabic{section}.\arabic{table}}

\renewcommand{\theenumi}{\roman{enumi}}

\def\min{{\mathrm{min}}}

\def\max{{\mathrm{max}}}
\def\sup{{\mathrm{sup}}}

\renewcommand{\abstractname}{$\square$}
\usepackage{mathpazo}

\begin{document}
\markboth{\fontsize{10}{0}Regularizing systems of accretive operator equations}{\fontsize{10}{0}P.K. Anh, N. Buong, and D.V. Hieu}%
\pagestyle{myheadings}
\setlength{\baselineskip}{15pt}

\title{\Large\bf\textsf{PARALLEL METHODS FOR REGULARIZING SYSTEMS\\ OF EQUATIONS INVOLVING ACCRETIVE OPERATORS}}
\author{ \small\bf{Pham Ky Anh\footnote{Corresponding author: Pham Ky Anh, Department of Mathematics, Hanoi University of Science, 334~ Nguyen~Trai, Thanh Xuan, Hanoi, Vietnam. Email address: phamkyanh@hus.edu.vn.}}
\\{\small\it Department of Mathematics, Hanoi University of Science, Hanoi, Vietnam}  
\\ \small\bf{ Nguyen Buong}
\\{\small\it Institute of Information Technology, Vietnamese Academy of Science \& Technology} 
\\ \small\bf{Dang Van Hieu  }
\\{\small\it Department of Mathematics, Hanoi University of Science, Hanoi, Vietnam} 
} 
\date{}
\maketitle
\parbox{15cm}{
\begin{abstract}{}
\medskip  
{\it In this paper, two parallel methods for solving systems of accretive operator equations in Banach spaces are studied. 
The convergence analysis of the methods in both free-noise and noisy data cases is provided.}\\
\medskip

\medskip

\end{abstract}}
\section{INTRODUCTION, PRELIMINARIES, AND NOTATIONS}
Various problems of science and engineering, including a multi-parameter identification problem, the convex feasibility problem, a common fixed point problem, etc..., lead to a system of ill-posed operator equations 
\begin{equation}\label{eq1}
A_i(x)=0,\quad x \in X, (i=1,2,\ldots,N),
\end{equation}
where $X$ is a real Banach space and  $A_i:D(A_i) = X \to X$ are possibly nonlinear operators on $X$.

Very recently, several sequential and parallel regularizing methods for solving system $(\ref{eq1})$ have been proposed.
The Kaczmarz method \cite{HLS2007,HKLS2007}, the Newton-Kacmarz method \cite{BK2006}, the steepest-descent-Kaczmarz method \cite{CHLS2008}, parallel iterative regularization methods \cite{AC2009}, parallel regularized Newton-type methods \cite{AC2011, AD2013}, parallel hybrid methods \cite{AC2013}, to name only few. However, most of the investigation of available methods was carried out in the framework of Hilbert spaces.

In this paper we study parallel methods extended to system $(\ref{eq1})$ involving $m$-accretive operators in the setting of  Banach spaces. In the sequel we always assume that system $(\ref{eq1})$ is consistent, i.e., the solution set $S$ of $(\ref{eq1})$ is not empty.  It is known that if $A_i (i=1,\ldots,N)$ are not strongly or uniformly accretive, then system $(\ref{eq1})$ in general is ill-posed, i.e., the solution set $S$ of $(\ref{eq1})$ may not depend continuously on data.  In that case, a process known as regularization should be applied for stable solution of $(\ref{eq1})$. 

In what follows, for  the reader's convenience,  we collect some definitions and results concerning the geometry of Banach spaces and accretive operators, which are used in this paper. We refer the reader to [9-13]
for more details. 

\begin{definition}\label{def.uniconvBspace}
A Banach space $X$ is called
\begin{itemize}
\item  [$1)$] strictly convex if the unit sphere $S_1(0) = \{x \in X: ||x||=1\}$ is strictly convex, i.e., the inequality $||x+y|| <2$ holds for all $x, y \in S_1(0),  x \ne y ;$ 
 \item [$2)$] uniformly convex if for any given $\epsilon >0$ there exists $\delta = \delta(\epsilon)  >0$ such that for all $x,y \in X$ with $\left\|x\right\| \le 1, \left\|y\right\|\le 1,\left\|x-y\right\| = \epsilon$ the inequality $\left\|x+y\right\|\le 2(1-\delta)$ holds.
\end{itemize}
\end{definition}
The modulus of convexity of $X$ is defined by
\begin{equation*}
\delta_{X}(\epsilon)=\inf \left\{1-\frac{\left\|x-y\right\|}{2}: \left\|x\right\| = \left\|y\right\|= 1,\left\|x-y\right\|=\epsilon\right\}.
\end{equation*}
The modulus of smoothness of $X$ is defined by
\begin{equation*}
\rho_X (\tau)=\sup\left\{\frac{\left\|x+y\right\|+\left\|x-y\right\|}{2}-1:\left\|x\right\|=1, \left\|y\right\|=\tau\right\}.
\end{equation*}
\begin{definition}\label{def.modusmooth}
A Banach space $X$ is called uniformly smooth if
\begin{equation*}
\lim_{\tau \to 0}h_X(\tau):=\lim_{\tau \to 0}\frac{\rho_X(\tau)}{\tau}=0.
\end{equation*}
\end{definition}
Observe that if $X$ is a real uniformly convex and uniformly smooth Banach space, then the modulus of convexity $\delta_X$ is a continuous and strictly increasing function on the whole segment $[0,2]$ (see, for example \cite{BCL1994}).\\
\begin{definition}
A Banach space $X$ possesses the approximation if there exists a directed family of finite dimensional subspaces $X_n$ ordered by inclusion, and a corresponding family of projectors $P_n: X \to X_n,$ such that $||P_n|| = 1$ for all $n \geq 0$ and $\cup_n X_n$ is dense in $X.$
\end{definition} 
Throughout this paper we assume that the so-called normalized duality mapping $J:X\to X^*,$ satisfying the relation
$$ \left\langle x, J\left(x\right)\right\rangle =\left\|x\right\|^2=\left\|J\left(x\right)\right\|^2, \quad \forall x \in X,$$
is single valued. This assumption will be fulfilled if $X$ is smooth.\\
For the sake of simpicity, we will denote norms of both spaces $X$ and $X^*$ by the same symbol $\left\|.\right\|$. The dual product of $f \in X^*$ and $x \in X$ will be denoted by $\left\langle x, f \right\rangle$ or $\left\langle f, x \right\rangle$. Besides, we put $ \mathbb R^+:= (0, \infty),    \mathbb R_*^+:= [0, \infty).$
\begin{definition}\label{def.accretive}
An operator $A:X \to X$ is called
\begin{itemize}
\item  [$1)$] accretive, if 
\begin{equation*}
\left\langle A(x)-A(y), J(x-y)\right\rangle  \ge 0\quad \forall x,y \in X;
\end{equation*}
\item [$2)$] maximal accretive, if it is accretive and its graph is not the right part of the graph of any other accretive operator;
\item [$3)$] $m$-accretive, if it is accretive and $R(A+\alpha I) = X$ for all $\alpha >0,$ where $I$ is the identity operator in $X;$
\item [$4)$] uniformly accretive, if there exists a strictly increasing function $\psi : \mathbb R_*^+  \to \mathbb R_*^+ ,  \psi(0) = 0,$ such that 
\begin {equation}\label{eq:unifaccre}
\left\langle A(x)-A(y), J(x-y)\right\rangle  \ge \psi(||x - y||) \quad \forall x,y \in X;
\end {equation}
\item [$5)$] strongly accretive, if there exists a positive constant $c,$ such that in $(\ref{eq:unifaccre})$, $\psi(t) = c t^2;$
\item [$6)$] inverse strongly accretive, if there exists a positive constant $c,$ such that
\begin{equation*}
\left\langle A(x)-A(y), J(x-y)\right\rangle  \ge c||A(x) - A(y)||^2 \quad \forall x,y \in X.
\end{equation*}
\end{itemize}
\end{definition}
If $X$ is a Hilbert space then $J$ is an identity operator and accretive operators are also called monotone.
\begin{definition}\label{def:quasiISA}
A continuous operator $A$ mapping a Banach space $X$ into itself is called $\varphi$-inverse uniformly accretive $($or  simply, inverse  uniformly  accretive $)$, if there exists a function $\varphi:\mathbb R^+ \times \mathbb R_*^+ \to \mathbb R_*^+$, which is continuous and strictly increasing in the second variable and $\varphi(s,t)=0$ if and only if $t=0$ for every fixed $s>0$, such that
\begin{equation}\label{eq:quasiISA}
\left\langle A(x)-A(y), J(x-y)\right\rangle  \ge \varphi \left(R,\left\|A(x)-A(y))\right\|\right) \forall x, y \in X, \left\|x\right\|,\left\|y\right\|\le R, \forall R >0.
\end{equation}
\end{definition}
\begin{example}\label{eg1}
{\rm Any inverse strongly accretive operator is inverse uniformly accretive, hence is accretive.
Indeed, let A be a $c$-inverse strongly accretive operator. Then A is Lipschitz continuous with the Lipschit constant $c^{-1}$ and the inequality $(\ref{eq:quasiISA})$ holds for the function $\varphi(s,t)= ct^2$}.
\end{example}
\begin{example}\label{eg2}
{\rm Let $T$ be a nonexpansive operator on a uniformly convex and uniformly smooth Banach space $X$. Then $A:=I-T$ is a Lipschitz continuous operator. Moreover, according to Alber \cite{A2000},
\begin{equation*}
\left\langle A(x)-A(y), J(x-y)\right\rangle  \ge L^{-1}R^2\delta_{X}\left(\frac{\left\|A(x)-A(y)\right\|}{4R}\right) \forall x, y \in X, \left\|x\right\|,\left\|y\right\|\le R,
\end{equation*}
where $L\in (1;1.7)$ is the Figiel constant and $\delta_X (\epsilon)$ is the modulus of the convexity of $X$. Observe that $\epsilon:=\frac{\left\|A(x)-A(y)\right\|}{4R} \le 1$ for any $x,y \in X;\left\|x\right\|,\left\|y\right\|\le R$ and inequality $(\ref{eq:quasiISA})$ holds for the function $\varphi(s,t)=L^{-1}s^2\delta_X\left(\frac{t}{4s}\right), s\in \mathbb R^+;   t\in [0;2s]$}.
\end{example}
\begin{example}\label{eg3}
{\rm Now let $X$ in Example 2 be one of the following Banach spaces $L^p,l^p, W^m_p$, where $1<p<\infty$. Then $X$ is uniformly smooth and uniformly convex, and it is well-known that (see \cite{AR2006})
\begin{align*}
&\delta_X(t) \ge \frac{p-1}{16}t^2, \quad 1< p <2; \\
&\delta_X(t) \ge \frac{1}{p2^p}t^p, \quad p\ge 2.
\end{align*}
Thus, for all $x, y \in X, \left\|x\right\|,\left\|y\right\|\le R$, one gets
\begin{align*}
& \left\langle A(x)-A(y), J(x-y)\right\rangle  \ge \frac{p-1}{256L}\left\|A(x)-A(y)\right\|^2, 1<p<2; \\
&\left\langle A(x)-A(y), J(x-y)\right\rangle  \ge \frac{1}{Lp 8^p}\frac{\left\|A(x)-A(y)\right\|^p}{R^{p-2}}, p \ge 2. 
\end{align*}
 So, the operator $A = I - T,$ where $T: l^p \to l^p$ is a nonexpansive operator, is inverse strongly accretive if $1<p<2$, and is inverse uniformly accretive with $\varphi (s,t) = \frac{t^p}{pL8^ps^{p-2}},$ if $p \ge 2$}.
\end{example}
\begin{definition}\label{def.weakcontinous}
An operator, $B:D(B) \subset X \to X$ is called
\begin{itemize}
\item [$1)$] hemicontinuous at a point $x_0 \in D(B)$, if $B(x_0+t_nh) \rightharpoonup x_0$ as $t_n \to 0$ for any vector $h$ such that $x_0+t_nh \in D(B)$ and $ 0 \leq t_n \leq t(x_0)$;
\item [$2)$] weakly continuous at $x_0 \in D(B)$, if $D(B)\ni x\rightharpoonup x_0$ implies that $B(x) \rightharpoonup B(x_0)$. 
\end{itemize}
\end{definition}
If $B$ is hemicontinuous (weakly continuous) at every point of $D(B)$, then $B$ is said to be hemicontinuous (weakly continuous), respectively.\\
For regularizing accretive operator equations one needs the following fact \cite{AR2006}.
\begin{lemma}\label{AlbRya2006}
Suppose that the Banach space $X$ possesses the approximation, $A: X \to X$ is a hemicontinuous accretive operator with $D(A) = X,$ and the normalized duality mapping $J: X \to X^*$ is sequentially weakly continuous and continuous. Then the problem 
\begin {equation}\label {eq:lem1}
A(x) + \alpha x = y,
\end{equation} 
where $\alpha$ is a fixed positive parameter and $y \in X$, is well-posed.
\end{lemma}
The unique solvability of $(\ref{eq:lem1})$ is established in \cite{AR2006}.  The continuous dependence of the solution $x_\alpha$ of $(\ref{eq:lem1})$ on the right-hand side $y$ follows from the inequality $||x_{\alpha, 1}-x_{\alpha, 2}|| \leq \frac{||y_1-y_2||}{\alpha},$ where $x_{\alpha, i}$ are the unique solution of  $(\ref{eq:lem1})$ with respect to the right-hand side $ y = y_i,  i =1,2$.\\

The next five lemmas will be used in Section 2 for establishing the convergence of implicit and explicit parallel iterative regularization methods.
\begin{lemma}\label{lem.ineqJ*}
{\rm \cite{A2007}}  Let $X$ be a real uniformly smooth Banach space. Then for any $x,y \in X $ such that $\left\|x\right\|\le R$, $\left\|y\right\| \le R$, the following inequality holds:
\begin{equation*}\label{eq:ineqJ*}
\left\|J(x)-J(y)\right\| \le 8R h_{X}\left(\frac{16L\left\|x-y\right\| }{R}\right),
\end{equation*}
where $L$ is Figiel constant, $(1<L<1.7)$.
\end{lemma}
\begin{lemma}\label{lem.ineqconvex}{\rm \cite{AR2006}}  If $X$ is a real uniformly smooth Banach space, then the inequality
\begin{align*}\label{eq:bdt.l}
\left\|x\right\|^2 &\le \left\|y\right\|^2 +2\left\langle x-y, J(x)\right\rangle\\
&\le \left\|y\right\|^2 +2\left\langle x-y, J(y)\right\rangle+2\left\langle x-y, J(x)-J(y)\right\rangle \notag
\end{align*}
holds for every $x,y \in X$.
\end{lemma}
\begin{lemma} {\rm \cite{AR2006}} \label{lem.geometry1}
Let X be a uniformly smooth Banach space. Then for $x,y\in X$
\begin{equation*}\label{eq:bd5}
\left\langle x-y, J(x)-J(y) \right\rangle \le 8\left\|x-y\right\|^2 +C\left(\left\|x\right\|, \left\|y\right\|\right) \rho_X (\left\|x-y\right\|),
\end{equation*}
where $C\left(\left\|x\right\|, \left\|y\right\|\right) \le 4$ max $\left\{2L, \left\|x\right\|+ \left\|y\right\|\right\}.$
\end{lemma}
\begin{lemma}{\rm \cite{AR2006}} \label{lem.geometry2}
In a uniformly smooth Banach space $X$, for $x,y \in X,$ 
\begin{equation*}\label{eq:bd6}
\left\langle x-y, J(x)-J(y) \right\rangle \le R^2 (\left\|x\right\|, \left\|y\right\|) \rho_X \left(\frac{4\left\|x-y\right\|}{R(\left\|x\right\|, \left\|y\right\|)}\right),
\end{equation*}
where $R(\left\|x\right\|, \left\|y\right\|)=\sqrt{2^{-1}(\left\|x\right\|^2+ \left\|y\right\|^2)}$.\\
If $\left\|x\right\|\le R, \left\|y\right\| \le R$, then 
\begin{equation*}\label{eq:bd6}
\left\langle x-y, J(x)-J(y) \right\rangle \le 2LR^2 \rho_X \left(\frac{4\left\|x-y\right\|}{R}\right).
\end{equation*}
\end{lemma}
\begin{lemma} \label{lem.sequence}{\rm \cite{AR2006, Xu2004}}
Let $\left\{\lambda_n\right\}$ and $\left\{p_n\right\}$ be sequences of nonnegative numbers, $\left\{b_n\right\}$ be a sequence of positive numbers, satisfying the inequalities
$$ \lambda_{n+1} \le \left(1-p_n\right)\lambda_n +b_n, \quad \forall n\ge 0, $$
where $p_n \in \left(0;1\right),\frac{b_n}{p_n}\to 0 \left(n\to +\infty \right)$ and  $\sum_{i=1}^\infty p_n =+\infty$. Then $\lambda_n\to 0\left(n\to +\infty\right)$.
\end{lemma}
In Section 3, when dealing with a parallel Newton-type regularization method, we need some more results.
\begin{lemma} {\rm \cite{Wang2008}} \label{lem.inequ.phu}
Suppose $A:D(A)=X\to X$ is  a continuously Fr$\acute{e}$chet differentiable accretive operator and let $L:=A^{'} (h), h \in X$, and $\alpha$ be a real positive number. Then
\begin{equation*}\label{eq: sa6}
\left\|\left(\alpha I +L\right)^{-1}\right\| \le \frac{1}{\alpha};  \left\|\left(\alpha I +L\right)^{-1}L\right\| \le 2.
\end{equation*}
\end{lemma}
\begin{lemma} {\rm \cite{BNS1997, AD2013}} \label{lem.sequence2}
Let $\left\{\omega_n\right\}$ be a sequence of nonnegative numbers satisfying the relations
\begin{equation*}\label{eq:lemma2}
\omega_{n+1}\le a+b\omega_n+c\omega_n^2, \quad n\ge 0,
\end{equation*}
for some $a,b,c>0$.\\ Let $M_{+}:=(1-b+\sqrt{(1-b)^2-4ac})/2c$, $M_{-}:=(1-b-\sqrt{(1-b)^2-4ac})/2c$. If $b+2\sqrt{ac}<1$ and $\omega_0 \le M_{+}$, then $\omega_n \le l:=$max$\left\{\omega_0, M_{-}\right\}$ for all $n\ge 0$.
\end{lemma}
An outline of the remainder of the paper is as follows: In Section 2 we propose two parallel iterative regularizations methods (PIRMs) for system $(\ref{eq1})$, namely implicit PIRM and explicit PIRM. The convergence of these PIRMs is established for both exact and noisy data cases. Section 3 studies a parallel regularizing Newton-type method for system $(\ref{eq1})$. The convergence analysis of the proposed method in exact and noisy data cases is also studied.
\section{EQUATIONS WITH INVERSE UNIFORMLY ACCRETIVE OPERATORS}
\setcounter{lemma}{0}
\setcounter{theorem}{0}
\setcounter{equation}{0}
In this section, we consider system $(\ref{eq1})$ with inverse uniformly accretive operators. Clearly, if each operator $A_i$ is $\varphi_i$-inverse uniformly accretive, then it is $\varphi$-inverse uniformly accretive with $\varphi (s,t): = \underset {i=1,..., N} \min \varphi_i(s,t).$ Thus, without loss of generality  we can assume that all the operators $A_i,   i=1, \ldots, N$ are $\varphi$-inverse uniformly accretive with the same funtion $\varphi.$\\ 
We begin with the following simple fact (cf. \cite{AC2011}).
\begin{lemma} \label{lem.iquivalent}
Suppose $A_i, i=1,2,\ldots, N,$ are inverse uniformly accretive operators. If system $(\ref{eq1})$ is consistent, then it is equivalent to the operator equation
\begin{equation}\label{eq:sum.equa}
A(x):=\sum_{i=1}^N A_i(x)=0.
\end{equation}
\end{lemma}
\begin{proof}
Let the opeartors $A_i,   i=1, \ldots, N,$ be $\varphi$-inverse uniformly accretive with the same funtion $\varphi.$  Obviously, any solution of  $(\ref{eq1})$ is a solution of $(\ref{eq:sum.equa})$. Conversely, let $y$ be a solution of $(\ref{eq:sum.equa}),$ i.e.,
\begin{equation*}\label{eq4}
\sum_{i=1}^N A_i(y)=0.
\end{equation*}
Let $z$ be a solution of  system $(\ref{eq1})$, i.e., $A_i(z)=0, i=1,2 \ldots N$. Then, $\sum_{i=1}^N (A_i(y)-A_i(z))=0,$ and since $A_i$ are inverse uniformly accretive, one gets
\begin{equation*}\label{eq:5}
\sum_{i=1}^N \varphi \left(R,\left\|A_i(y)-A_i(z)\right\|\right) \le \sum_{i=1}^N \left\langle A_i(y)-A_i(z),J(y-z)\right\rangle =0,
\end{equation*}
where $R=\max\left\{\left\|y\right\|,\left\|z\right\|\right\}$. Thus $\varphi \left(R,\left\|A_i(y)-A_i(z)\right\|\right)= 0$, hence
$A_i(y)=A_i(z)=0, \quad i=1,2,\ldots N$. Therefore, $y$ is a solution of system $(\ref{eq1})$. 
\end{proof}
In this section we need the following result .
\begin{lemma}\label{buong}{\rm (\cite{BP2012},Theorem 2.1)}.
Let $X$ be a real, reflexive and strictly convex Banach space with a uniformly Gateaux differentiable norm and let $A$ be an $m$-accretive mapping on $X$. Then for each $\alpha >0$ and a fixed $y \in X$, equation $(\ref{eq:lem1})$ possesses a unique solution $x_\alpha$, and in addition, if the solution set $S_A$ of the equation $A(x) = y$ is nonempty, then the net $\{x_\alpha\}$ converges strongly to the unique element $\widehat{x}^*$ solving the following variational inequality
\begin{equation*}
\left\langle \widehat{x}^*,J(\widehat{x}^*-x^*)\right\rangle \le 0, \forall x^* \in S_A.
\end{equation*}
Moreover we have $||x_\alpha^\delta - x_\alpha|| \leq \delta/\alpha,$ where $x_\alpha^\delta$ is the unique solution of the equation $A(x) + \alpha x = y_\delta,$ for any $\alpha >0$ and $y_\delta \in X$ satisfying $||y_\delta - y || \leq \delta.$  
\end{lemma}
In the remainder of Section 2, we impose two sets of conditions on the space $X$, the duality mapping $J$, and the operators $A_i, i=1,2,\ldots,N.$\\
{\bf Conditions (AJX)}
\begin{itemize}
\item [A1.] $A_i, i=1,2,\ldots,N,$ {\it are $\varphi$-inverse uniformly accretive operators with $D(A_i)=X$;}
\item  [A2.]{\it The normalized duality mapping $J$ is sequentially weakly continuous and continuous;} 
\item [A3.] {\it $X$ is a smooth and reflexive Banach space, possessing the approximation.}
\end{itemize}
{\bf Conditions (AX)}
\begin{itemize}
\item [B1.]   $A_i, i=1,2,\ldots,N,$ {\it are  $m$-accretive and $\varphi$-inverse uniformly accretive operators with} $D(A_i)=X;$ 
\item [B2.] $X$ {\it is a uniformly smooth and uniformly convex Banach space.}
\end{itemize}
Together with equation $(\ref{eq:sum.equa})$ we consider the following regularized one
\begin{equation}\label{eq:regu.sum.equ}
A(x)+\alpha_n x=\sum_{i=1}^NA_i(x)+\alpha_n x=0.
\end{equation}
\begin{lemma}\label{lem.regu.solution}
Let conditions {\rm A1-A3}  or  {\rm B1-B2}   be fulfilled. Then the following statements hold:
\begin{enumerate}
\item [{\rm i)}] For every $\alpha_n >0$, equation $(\ref{eq:regu.sum.equ})$ has a unique solution $x_n^*$.
\item [{\rm ii)}] $\left\|x_n^*\right\| \le 2\left\|\widehat{x}\right\|$, where $\widehat{x}$ is an arbitrary element of $S$.
\item [{\rm iii)}] $x_n^* \to \widehat{x}^*$ as $n \to+\infty$, where $\widehat{x}^*$ is an unique solution of the inequality $\left\langle \widehat{x}^*,J(\widehat{x}^*-x^*)\right\rangle \le 0, \forall x^* \in S.$
\item [{\rm iv)}] $\left\|x_n^*-x_{n+1}^*\right\|\le 2\left\|\widehat{x}^*\right\|\frac{\left|\alpha_{n+1}-\alpha_n\right|}{\alpha_n}.$\\
\item[{\rm v)}] $\left\|A_i(x_n^*)\right\| \le \varphi_{R}^{-1}\left(6\alpha_n\left\|\widehat{x}^*\right\|^2\right) \quad i=1,2,\ldots,N$, where $R >0$ is a fixed number satisfying an a-priori estimate $R \geq 2\left\|\widehat{x}^*\right\|$ and $\varphi_{s}^{-1}$ denotes the inverse function of $\varphi \left(s,t\right)$ with respect to the second variable $t$ for fixed $s>0$.
\end{enumerate}
\end{lemma}
\begin{proof}
1. Suppose that conditions A1-A3 hold. We perform the regularization process $(\ref{eq:regu.sum.equ})$ for equation $(\ref{eq:sum.equa})$ with the accretive operator $A = \sum_{i=1}^N A_i.$  For the proofs of  statements $i)-iv)$ we refer the reader to \cite{AR2006}. Concerning the last part $v)$  we observe that 
$$ \sum_{i=1}^N (A_i(x_n^*)-A_i(\widehat{x}^*)) +\alpha_n x_n^*=0,$$
hence, $$ \sum_{i=1}^N\left\langle A_i(x_n^*)-A_i(\widehat{x}^*),J(x_n^*-\widehat{x}^*)\right\rangle +\alpha_n \left\langle x_n^*,J(x_n^*-\widehat{x}^*)\right\rangle = 0.$$
Observing that by part ii) $\left\|x_n^*\right\| \le 2\left\|\widehat{x}^*\right\|,$ hence, $\left\|x_n^*-\widehat{x}^*\right\| \le 3\left\|\widehat{x}^*\right\|$ and using the inverse uniform accretiveness of $A_i$, from the last inequality we have
\begin{equation*}\label{eq:11}
\sum_{i=1}^N\varphi\left(R,\left\|A_i(x_n^*)-A_i(\widehat{x}^*)\right\|\right) \le  -\alpha_n \left\langle x_n^*,J(x_n^*-\widehat{x}^*)\right\rangle \le \alpha_n \left|| x_n^*\right||.\left||x_n^*-\widehat{x}^*\right||,
\end{equation*}
where $R \geq 2\left\|\widehat{x}^*\right\|$.
The last inequality gives $\varphi\left(R,\left\|A_i(x_n^*)\right\|\right) \le 6\alpha_n \left\|\widehat{x}^*\right\|^2$. Thus $\left\|A_i(x_n^*)\right\| \le \varphi_{R}^{-1}\left(6\alpha_n\left\|\widehat{x}^*\right\|^2\right)$.\\
2. Now suppose that conditions B1-B2 hold. Since all $A_i$ are inverse uniformly accretive, they are continuous, hence locally bounded. Besides, $A_i,  i=1,\ldots N,$ are $m$-accretive, $D(A_i) = X,$ and the spaces $X$ and $X^*$ are uniformly convex, then by  Theorem 1.15.22 (\cite{AR2006}),  the operator $A = \sum_{i=1}^N A_i$  is also $m$-accretive. Lemma  $\ref{buong}$ applied to equation $(\ref{eq:regu.sum.equ})$ ensures the convergence of regularized solutions $x_n^*$ to $\widehat{x}^*.$ The remaining statements can be argued similarly as in part 1.
\end{proof}
Following \cite{AC2009} we consider an implicit PIRM consisting of solving simultaneously $N$ regularized equations
\begin {equation}\label{eq:12}
A_i(x_n^i)+(\frac{\alpha_n}{N}+\gamma_n)x_n^i=\gamma_n x_n, i=1,2,\ldots,N,
\end{equation}
where $\alpha_n >0$ and $\gamma_n> 0$ are regularization and parallel splitting up parameters, respectively, and defining the next approximation as an average of the regularized solutions $x_n^i$,
\begin{equation}\label{eq:13}
x_{n+1}=\frac{1}{N}\sum_{i=1}^Nx_n^i, \quad  n= 0,1, \ldots, \quad x_0 \in X. 
\end{equation}
According to Lemmas $\ref{buong}$, $\ref{lem.regu.solution}$,  all the problems $(\ref{eq:12})$ are well posed and independent from each other, hence the regularized solutions $x_n^i$ can be found stably and simultaneously by parallel processors.\\
We first prove the boundedness of the sequence  $\left\{x_n\right\}$ defined by the implicit PIRM $(\ref{eq:12})$-$(\ref{eq:13})$.
\begin{lemma} \label{lemma.imp.boundedness}
Under conditions {\rm A1-A3} or {\rm B1-B2}, the sequence $\left\{x_n\right\}$ generated by $(\ref{eq:12})$ and $(\ref{eq:13})$ is bounded.
\end{lemma}
\begin{proof}By Lemma $\ref{AlbRya2006}$ (Lemma $\ref{buong}$), the regularized equation $(\ref{eq:12})$ has a unique solution denoted by $x_n^i$. Let $B_r(\widehat{x}^*)$ be the closed ball with center $\widehat{x}^*$and radius $r$. Choose $r>0$ sufficiently large such that $r\ge \left\|\widehat{x}^*\right\|$ and $x_0 \in B_r(\widehat{x}^*)$. Supposing for some $n>0$,  $x_n \in B_r(\widehat{x}^*)$, we will show that $x_{n+1} \in B_r(\widehat{x}^*)$. Indeed, from $(\ref{eq:12})$ and $A_i(\widehat{x}^*)=0$, we get 
\begin{equation*}\label{eq:14}
(A_i(x_n^i)-A_i(\widehat{x}^*))+(\frac{\alpha_n}{N}+\gamma_n)(x_n^i-\widehat{x}^*)=\gamma_n(x_n-\widehat{x}^*)-\frac{\alpha_n}{N}\widehat{x}^*.
\end{equation*}
Thus
\begin{align*}
\left\langle A_i(x_n^i)-A_i(\widehat{x}^*),J(x_n^i-\widehat{x}^*)\right\rangle &+(\frac{\alpha_n}{N}+\gamma_n)\left\langle x_n^i-\widehat{x}^*,J(x_n^i-\widehat{x}^*)\right\rangle\\ 
& =\gamma_n\left\langle x_n-\widehat{x}^*,J(x_n^i-\widehat{x}^*) \right\rangle -\frac{\alpha_n}{N}\left\langle \widehat{x}^*,J(x_n^i-\widehat{x}^*)\right\rangle.
\end{align*}
By the accretiveness of $A_i$, we get 
$$ (\frac{\alpha_n}{N}+\gamma_n)\left\|x_n^i-\widehat{x}^*\right\|^2 \le \gamma_n \left\|x_n-\widehat{x}^*\right\| \left\|x_n^i-\widehat{x}^*\right\|+\frac{\alpha_n}{N}\left\|\widehat{x}^*\right\| \left\|x_n^i-\widehat{x}^*\right\|,$$
hence
$$ (\frac{\alpha_n}{N}+\gamma_n)\left\|x_n^i-\widehat{x}^*\right\| \le \gamma_n \left\|x_n-\widehat{x}^*\right\|+\frac{\alpha_n}{N}\left\|\widehat{x}^*\right\|.$$
Using the inequalities $\left\|x_n-\widehat{x}^*\right\| \le r$ and $r\ge \left\|\widehat{x}^*\right\|$, we have 
$$ (\frac{\alpha_n}{N}+\gamma_n)\left\|x_n^i-\widehat{x}^*\right\| \le \gamma_n r+\frac{\alpha_n}{N}r \le (\frac{\alpha_n}{N}+\gamma_n)r ,$$ which gives
$\left\|x_n^i-\widehat{x}^*\right\| \le r$.
By $(\ref{eq:13})$, one gets 
$$ \left\|x_{n+1}-\widehat{x}^*\right\| \le \frac{1}{N}\sum_{i=1}^{N}\left\|x_n^i-\widehat{x}^*\right\| \le r .$$
Therefore, $x_{n+1}\in B_r(\widehat{x}^*)$. Thus, $\left\{x_n\right\}$ is bounded.
\end{proof}

\begin{theorem}\label{theo.imp.convergence}
Suppose conditions  {\rm A1-A3} or {\rm B1-B2} are fulfilled. Let $\left\{\alpha_n \right\}$ and $\left\{\gamma_n\right\}$ be real sequences, such that
\begin{enumerate}
\item [{\rm i)}] $\alpha_n \to 0, \gamma_n \to +\infty$ as $n\to +\infty$,
\item [{\rm ii)}] $\frac{\gamma_n |\alpha_{n+1}-\alpha_{n}|} {\alpha_n^2} \to 0$ as $n\to +\infty$, $\sum_{n=1}^{\infty}\frac{\alpha_n}{\gamma_n}=+\infty$,
\item [{\rm iii)}]$\frac{h_{X}(\tau_n)\varphi_R^{-1}\left(R_1\alpha_n\right)}{\alpha _n}\to 0$ as $n\to +\infty$, where $R \geq 2||\widehat{x}^*||,  R_1:= \frac{3R^2}{2} $ and $\tau_n =\gamma_n^{-1}$.
\end{enumerate}
If in addition, the function $\frac{\varphi(s,t)}{t}$ is coercive in $t$ for any fixed $s>0,$ i.e., $ \frac{\varphi(s,t)}{t} \to +\infty$ as $t \to +\infty,$ then starting from arbitrary $x_0\in X$, the sequence $\left\{x_n\right\}$ defined by $(\ref{eq:12})$ and $(\ref{eq:13})$ converges strongly to $\widehat{x}^*$.
\end{theorem}
\begin{proof}
Let $x_n^{*}$ be the unique solution of $(\ref{eq:regu.sum.equ})$.  Setting $e_n^i=x_n^i-x_n^{*};e_n=x_n-x_n^{*}; \epsilon_n=\frac{\alpha_n\tau_n}{N}$, we can rewrite $(\ref{eq:12})$ as
\begin{equation*}\label{eq:15}
x_n^i+\tau_n A_i(x_n^i)+\epsilon_n x_n^i=x_n,
\end{equation*}
or
\begin{equation*}\label{eq:16}
\left(e_n^i-e_n\right)+\tau_n \left[{A}_i(x_n^i) -{A}_i(x_n^{*})\right] +\epsilon_n e_n^i= -\tau_n {A}_i(x_n^{*})-\epsilon_n x_n^*, \quad i=1,2,\ldots N.
\end{equation*}
From the last relation, using the accretiveness of $A_i$ we find
\begin{equation}\label{eq:17}
2\left\langle \left(e_n^i-e_n\right), J(e_n^i)\right\rangle +2\epsilon_n\left\| e_n^i \right\|^2\le -2\left\langle \tau_n {A}_i(x_n^{*})+\epsilon_n x_n^*,J(e_n^i)\right\rangle.
\end{equation}
From Lemma $\ref{lem.ineqconvex}$, we get 
\begin{equation*}\label{eq:18}
2\left\langle e_n^i-e_n,J(e_n^i) \right\rangle \ge \left\| e_n^i \right\|^2 -\left\| e_n \right\|^2.
\end{equation*}
Combining this inequality with $(\ref{eq:17})$, we obtain
\begin{equation*}\label{eq:19}
\left(1+2\epsilon_n\right)\left\| e_n^i \right\|^2-\left\| e_n \right\|^2 \le -2\tau_n\left\langle  {A}_i(x_n^{*})+\frac{\alpha_n}{N} x_n^*,J(e_n^i)\right\rangle,
\end{equation*}
hence,
\begin{equation}\label{eq:19*}
\left(1+2\epsilon_n\right)\sum_{i=1}^N\left\| e_n^i \right\|^2-N\left\| e_n \right\|^2 \le -2\tau_n \sum_{i=1}^N\left\langle{A}_i(x_n^{*})+\frac{\alpha_n}{N} x_n^*,J(e_n^i)\right\rangle.
\end{equation}
Observing that  $x_n^*$ is the solution of $(\ref{eq:regu.sum.equ})$ and using Lemma $\ref{lem.regu.solution}$, we can estimate the right-hand side of  $(\ref{eq:19*})$ as follows
\begin{align} \label {eq:22}
-\sum_{i=1}^N\langle{A}_i(x_n^{*}) \notag &+\frac{\alpha_n}{N} x_n^*,J(e_n^i)\rangle = -\left\langle \sum_{i=1}^{N}{A}_i(x_n^{*})+\alpha_n x_n^* ,J(e_n) \right\rangle\\\notag
&-\sum_{i=1}^{N}\left\langle {A}_i(x_n^{*})+\frac{\alpha_n}{N} x_n^* ,J(e_n^i)-J(e_n) \right\rangle\\  \notag
&\le  \sum_{i=1}^{N} \left(\left\| {A}_i(x_n^{*})\right\|+ \frac{\alpha_n}{N} \left\|x_n^* \right\|\right) \left\| J(e_n^i)-J(e_n)  \right\| \\ 
&\le  \sum_{i=1}^{N} \left(\varphi_R^{-1}\left(6\alpha_n \left\| \widehat{x}^* \right\|^2\right)+\frac{2\alpha_n}{N} \left\|\widehat{x}^* \right\|\right) \left\| J(e_n^i)-J(e_n)  \right\|.
\end{align}
By Lemma $\ref{lem.regu.solution}$ and Lemma $\ref{lemma.imp.boundedness}$, the sequences $\left\{x_n^*\right\}$ , $\left\{x_n \right\}$ and $\left\{x_n^i\right\}$ are bounded, hence the sequences $\left\{e_n \right\}$ and $\left\{e_n^i\right\}$ are also bounded, i.e., there exists a positive constant $C>0$ such that $\left\| e_n \right\|\le C; \left\| e_n^i \right\| \le C;\left\| x_n^i \right\| \le C$. \\
From Lemma $\ref{lem.ineqJ*}$, we get
\begin{equation}\label{eq:23}
\left\|J(e_n^i)-J(e_n)\right\| \le 8C h_{X}\left(\frac{16L\left\|x_n^i-x_n\right\|}{C}\right),
\end{equation}
where $L \in (1, 1.7)$ is Figiel constant. \\
We show that $||A_i(z)|| \leq C_i < +\infty$ for all $||z|| \leq R_0:= C + 2||\widehat {x^*}||$  and $i =1,\ldots,N.$ Indeed, suppose in contrary, that there exists a sequence $\{z_n\}$, such that $||z_n|| \leq R_0,$ and $||A_i(z_n)|| \to \infty$ as $n \to \infty.$ Then $t_n := ||A_i(z_n) - A_i(0)|| \geq ||A_i(z_n)|| - ||A_i(0)|| \to \infty$ as $n \to \infty.$ Since $\varphi(R_0, t_n) = \varphi(R_0, ||A_i(z_n)- A_i(0)||) \leq \left\langle {A}_i(z_n) - {A}_i(0), J(z_n - 0) \right\rangle \leq ||A_i(z_n) - A_i(0)||||z_n|| \leq R_0 t_n,$ we get $R_0 \geq \frac{\varphi(R_0,t_n)}{t_n},$ which contradicts the coerciveness of $\frac{\varphi(R_0,t)}{t} .$ \\
Thus, we can put 
$$ M=\sup \left\{\left\|A_i(x)+\frac{\alpha_n}{N}x \right\|: \left\|x\right\| \le R_0 , n=1,2,\ldots, i=1,2,\ldots,N\right\}. $$
Relation $(\ref{eq:12})$ yields
$$ A_i(x_n^i)+\frac{\alpha_n}{N}x_n^i = \gamma_n \left(x_n -x_n^i\right), i=1,2,\ldots,N, $$
which gives $\gamma_n \left\|x_n -x_n^i\right\| \le \left\|A_i(x_n^i)+\frac{\alpha_n}{N}x_n^i\right\| \le M,$ hence, $\left\|x_n -x_n^i\right\| \le \frac{M}{\gamma_n}=M\tau_n$. Combining the last inequality with $(\ref{eq:23})$, we obtain 
\begin{equation}\label{eq:24}
\left\|J(e_n^i)-J(e_n)\right\| \le c_2 h_{X}(k_0\tau_n ),
\end{equation}
where $c_2=8C, k_0=\frac{16LM}{C}$.
By $(\ref{eq:22})$ and $(\ref{eq:24})$, we have
\begin{equation}\label{eq:25}
-\sum_{i=1}^N\left\langle{A}_i(x_n^{*})+\frac{\alpha_n}{N} x_n^*,J(e_n^i)\right\rangle \le Nc_2\left(\varphi_R^{-1}\left(6\alpha_n \left\| \widehat{x}^* \right\|^2\right)+\frac{2\alpha_n}{N} \left\|\widehat{x}^* \right\|\right)h_{X}(k_0\tau_n ).
\end{equation}
From $(\ref{eq:19*}), (\ref{eq:25})$, we get
\begin{equation}\label{eq:26}
(1+2\epsilon_n)\sum_{i=1}^{N}\left\| e_n^i \right\|^2 \le N\left\| e_n \right\|^2+2Nc_2 \tau_n \left(\varphi_R^{-1}\left(6\alpha_n \left\| \widehat{x}^* \right\|^2\right)+\frac{2\alpha_n}{N} \left\|\widehat{x}^* \right\|\right)h_{X}(k_0\tau_n ).
\end{equation}
Taking into account relation $(\ref{eq:13})$, Lemma $\ref{lem.regu.solution}$, and the inequality $(a+b)^2 \le (1+\epsilon_n)(a^2 +\frac{b^2}{\epsilon_n})$, we find
\begin{align*}\label{eq:27}
\left\| e_{n+1} \right\|^2 &\notag=\left\| x_{n+1}-x_{n+1}^* \right\|^2 \le \left(\left\| x_{n+1}-x_{n}^* \right\|+\left\| x_{n}^*-x_{n+1}^* \right\|\right)^2\\ \notag
&\le \left(\left\| x_{n+1}-x_{n}^* \right\|+2\left\|\widehat{x}^*\right\|\frac{|\alpha_{n+1}-\alpha_n|}{\alpha_n}\right)^2\\ \notag
&\le\left( \frac{1}{N}\sum_{i=1}^{N}\left\| e_n^i \right\|+ 2\left\|\widehat{x}^*\right\|\frac{|\alpha_{n+1}-\alpha_n|}{\alpha_n}\right)^2\\ \notag
&\le \left( \frac{1}{\sqrt{N}}\left(\sum_{i=1}^{N}\left\| e_n^i \right\|^2\right)^{1/2} + 2\left\|\widehat{x}^*\right\|\frac{|\alpha_{n+1}-\alpha_n|}{\alpha_n}\right)^2\\ \notag
&\le \left(1+\epsilon_n\right)\left(  \frac{1}{N}\sum_{i=1}^{N}\left\| e_n^i \right\|^2 +4\left\|\widehat{x}^*\right\|^2\frac{|\alpha_{n+1}-\alpha_n|^2}{\alpha_n^2  \epsilon_n}\right).
\end{align*}
Thus,
\begin{equation}\label{eq:28}
\frac{N}{\left(1+\epsilon_n\right)}\left\| e_{n+1} \right\|^2 -4N\epsilon_n \left(\frac{\alpha_{n+1}-\alpha_n}{\epsilon_n  \alpha_n}\right)^2 \left\|\widehat{x}^*\right\|^2 \le \sum_{i=1}^{N}\left\| e_n^i \right\|^2.
\end{equation} 
From $(\ref{eq:28}), (\ref{eq:26})$, one gets
\begin{align}\label{eq:29}
\left\| e_{n+1} \right\|^2 &\notag\le \frac{1+\epsilon_n}{1+2\epsilon_n}\left\| e_{n} \right\|^2+ 4(1+\epsilon_n) \epsilon_n \left(\frac{\alpha_{n+1}-\alpha_n}{\epsilon_n\alpha_n}\right)^2\left\| \widehat{x}^* \right\|^2 \\ 
&  +\frac{2c_2\left(1+\epsilon_n\right) \tau_n}{1+2\epsilon_n}\left(\varphi_R^{-1}\left(6\alpha_n \left\| \widehat{x}^* \right\|^2\right)+\frac{2\alpha_n}{N} \left\|\widehat{x}^* \right\|\right)h_{X}(k_0\tau_n ).
\end{align}
Setting $\lambda_n=\left\| e_{n} \right\|^2; p_n=\frac{\epsilon_n}{1+2\epsilon_n}$ and $b_n=b_{1n}+b_{2n}+b_{3n}$, where 
\begin{align*}
&b_{1n}=\frac{4(1+\epsilon_n)}{\epsilon_n} \left(\frac{\alpha_{n+1}-\alpha_n}{\alpha_n}\right)^2\left\| \widehat{x}^* \right\|^2; \\
&b_{2n}=\frac{2c_2\left(1+\epsilon_n\right) \tau_n}{1+2\epsilon_n}\varphi_R^{-1}\left(6\alpha_n \left\| \widehat{x}^* \right\|^2\right)h_{X}(k_0\tau_n );\\
&b_{3n}=\frac{4c_2\left(1+\epsilon_n\right) \tau_n\alpha_n}{N(1+2\epsilon_n)} \left\|\widehat{x}^* \right\|h_{X}(k_0\tau_n ).
\end{align*}
We can rewrite $(\ref{eq:29})$ as $\lambda_{n+1}\le \left(1-p_n\right)\lambda_{n}+b_n$. Clearly, $\lambda_n , b_n \ge 0; p_n \in \left(0;1\right)$ and $p_n \to 0$ as $n \to +\infty$.

Since $p_n=\frac{\epsilon_n}{1+2\epsilon_n}$ and $\epsilon_n \to 0$ as $n\to+\infty$, the series $\sum_{n=1}^{\infty} p_n = +\infty$ if and only if $\sum_{n=1}^{\infty} \epsilon_n = +\infty$. The last fact is equivalent to the assumption $\sum_{n=1}^{\infty}\frac{\alpha_n}{\gamma_n}=+\infty$.\\
 By the assumption ii), $\frac{b_{1n}}{p_n}= 4\left(1+\epsilon_n\right)\left(1+2\epsilon_n\right)\left(\frac{\alpha_{n+1}-\alpha_n}{\epsilon_n \alpha_n}\right)^2\left\| \widehat{x}^* \right\|^2 \to 0$ as $n\to +\infty$. \\
Further, using assumption iii), we will show that the expression
\begin{equation}\label{eq:28a}
 \frac{b_{2n}}{p_n}=2c_2N\left(1+\epsilon_n\right)\frac{\varphi_R^{-1}\left(6\alpha_n \left\| \widehat{x}^* \right\|^2\right)h_{X}(k_0\tau_n )}{\alpha_n}
\end{equation}
will tend to zero as $n\to +\infty$. We first prove that there exist positive integers $m$ and $n_0$, such that for all $n \geq n_0,   h_X(k_0\tau_n) \leq \frac{5^m}{k_0} h_X(\tau_n).$  Indeed, according to [10, Lemma 1, page 65], we have
$$2 \leq \lim_{\tau \to 0^+}\sup\frac{\rho_X(2\tau)}{\rho_X(\tau)} \leq 4. $$
Hence, there exists $\tau_0 > 0,$ such that $\frac{\rho_X(2\tau)}{\rho_X(\tau)} \leq 5$ for all $\tau \leq \tau_0.$ 
Since $\tau_n \to 0$ as $n \to + \infty,$ we can find a number $n_0$ such that $k_0 \tau_n \leq \tau_0$ for all $n \geq n_0.$ Let $m$ be a sufficiently large positive integer, such that $2^m \geq k_0.$ Then for all $n \geq n_0$ we have $\rho_X(k_0 \tau_n) = \rho_X(2 \frac{k_0 \tau_n}{2^1}) \leq 5 \rho_X(  \frac{k_0 \tau_n}{2^1}) = 5 \rho_X( 2 \frac {k_0 \tau_n}{2^2}) \leq 5^2 \rho_X( \frac {k_0 \tau_n}{2^2})  \leq \ldots \leq 5^m \rho_X( \frac{k_0 \tau_n}{2^m}).$ Because of the convexity of $\rho_X$ and $\frac{k_0}{2^m} \leq 1,$ we get $\rho_X (k_0 \tau_n) \leq 5^m \rho_X(\frac{k_0 \tau_n}{2^m}) \leq 5^m \rho_X(\tau_n).$ Thus, we come to the relation $h_X(k_0 \tau_n) = \frac {\rho_X(k_0 \tau_n)}{k_0 \tau_n} \leq \frac{5^m}{k_0} \frac {\rho_X(\tau_n)}{\tau_n}= \frac {5^m}{k_0}h_X(\tau_n).$\\
 Now using the last inequality and taking into account the fact that $\varphi_R^{-1}(t)$ is an increasing function and $R_1 := \frac{3}{2}R^2 \geq 6||\widehat{x}^*||^2,$ we can estimate the expression  $(\ref{eq:28a})$ as $\frac{b_{2n}}{p_n} \leq \frac{2c_2N 5^m (1 +\epsilon_n)}{k_0} \frac{\varphi_R^{-1}(R_1 \alpha_n) h_X(\tau_n)}{\alpha_n}$ for all $n \geq n_0.$ The assumption iii) implies that $\frac{b_{2n}}{p_n} \to  0$ as $n \to +\infty.$
\\Finally, the uniform smoothness of $X$ gives $$\frac{b_{3n}}{p_n}=4c_2\left(1+\epsilon_n\right)\left\|\widehat{x}^* \right\|h_{X}(k_0\tau_n )\to 0$$ as $n\to +\infty$. Thus, $\frac{b_n}{p_n} \to 0 (n\to +\infty)$. Lemma $\ref{lem.sequence}$ ensures that $\lambda_n =\left\| e_n \right\|^2 =\left\|x_n - x_n^*\right\|^2 \to 0 (n\to +\infty)$. Besides, by Lemma $\ref{lem.regu.solution}$, $x_n^* \to \widehat{x}^* (n\to +\infty)$, hence $\left\| x_n - \widehat{x}^* \right\| \le \left\|x_n - x_n^*\right\| + \left\|x_n^* -\widehat{x}^* \right\| \to 0 (n\to +\infty)$. The proof of Theorem $\ref{theo.imp.convergence}$ is complete.
\end{proof}
\begin{example}\label{eg4}
{\rm Let $A_i,    i=1,\ldots,N$ be $c-$ inverse strongly monotone operators on a real Hilbert space $X$. Then all the conditions A1-A3 and B1-B2 are satisfied. Further, since $\varphi(s,t) = ct^2,$ the function $\frac{\varphi(s,t)}{t}= ct$ is coercive. Conditions i), ii) on the parameters $\alpha_n,  \gamma_n $ have been already stated in [5, Theorem 2.1].
On the other hand,  for a Hilbert space,    $\rho_{X}(t)=\sqrt{1+t^2}-1\le \frac{t^2}{2}$, hence the assumption  iii) of Theorem $\ref{theo.imp.convergence}$ leads to the additional constraint $\gamma_n\alpha_n^{1/2} \to +\infty \quad (n \to +\infty).$}
\end{example}
An example of such a pair of parameters could be $\alpha_n = (n+1)^{-p},$ where, $0<p<1/2$ and $ \gamma_n = (n+1)^{1/2}.$ \\
In the next two examples we suppose that  $X= l^p,  1 \leq p < +\infty$ and  $ A_i = I - T_i,$ where $T_i : X \to X,     i=1,\ldots,N,$ are nonexpansive operators. In this case both sets of conditions A1-A3 and B1-B2 are fulfilled. Observe that for proving the $m$- accretiveness of $A_i$ one shoud use the identity $A_i + \alpha I = (1+\alpha)\{ I - (1+\alpha)^{-1}T_i\}$ and the fact that $(1+\alpha)^{-1}T_i$ is a contraction for $i=1,\ldots,N.$
\begin{example}\label{eg5}
{\rm Let $X= l^p$ with $p \geq 2,$ then $\rho_{X}(t)\le (p-1)t^2$ and $h_{X}(t)\le (p-1)t$ \quad (see \cite{AR2006}, page 48). According to Example $\ref{eg3}$, all the operators $A_i:=I-T_i,i=1,2,\ldots,N,$ are inverse uniformly accretive with $\varphi(s,t)=\frac{1}{Lp8^p}\frac{t^p}{s^{p-2}}. $  For any fixed $s>0,  \varphi_s^{-1}(t)=c(s)t^{\frac{1}{p}},$  where $c(s)$ is a positive constant,  and  the function $\frac{\varphi(s,t)}{t}$ is coercive in $t$. The assumption iii) of Theorem $\ref{theo.imp.convergence}$ becomes $\gamma_n \alpha_n^{\frac{p-1}{p}}\to \infty\quad (n\to +\infty)$ and we can choose $\alpha_n=(n+1)^{-k},\gamma_n=(n+1)^{1/2}$ with $0<k<\frac{1}{2}$.}
\end{example}
\begin{example}\label{eg6}
{\rm Suppose $X = l^p,     1<p < 2$, then we have (see \cite{AR2006}, page 48)
$\rho_X (t) \le \frac{t^{p}}{p}, h_{X}\left(t\right) \le \frac{t^{p-1}}{p}.$ Example $\ref{eg3}$ shows that $A_i,  i=1,\ldots, N$ are $c-$ inverse strongly accretive operators with $\varphi(s,t)=ct^2,   \varphi_s^{-1}(t)=\frac{\sqrt{t}}{\sqrt{c}},  \quad  c=\frac{p-1}{256L}$ 
and the assumption iii) of Theorem $\ref{theo.imp.convergence}$ becomes 
$\alpha_n^{1/2}\gamma_n^{p-1} \to \infty\quad (n\to +\infty).$ We can chose 
$\alpha_n=(n+1)^{-k},\gamma_n=(n+1)^{1/2}$ with $0<k<\min \{\frac{1}{2},p-1\}$.}
\end{example}
\indent Next we turn to the noisy data case. Assume that $A_i(x) :=F_i(x)-f_i$ and the exact operators $F_i(x),   i=1,\ldots, N,$ are inverse uniformly accretive. Suppose that instead of the exact data $(F_i,f_i)$, we have only noisy ones $(F_{n,i},f_{n,i})$, where the perturbed operators  $F_{n,i}:D(F_{n,i})=X\to X$ are just accretive for all $ n \geq 1$ and $i = 1, \ldots, N.$ Moreover, let
\begin{align}
&\left\|F_{n,i}(x)-F_i (x)\right\|\le h_n g(\left \|x\right\|) \label{eq:36},\\ 
& \left \|f_{n,i}-f_i\right\| \le \delta_n,  \quad i=1,\ldots,N \label{eq:37},
\end{align}
where, $g(t)$ is a nonnegative continuous nondecreasing function, $h_n>0,\delta_n>0$ for all $n>0$. Starting from arbitrary    $z_0\in X$,   we perform the following implicit PIRM:
\begin{align}
&A_{n,i}(z_n^i)+(\frac{\alpha_n}{N}+\gamma_n)z_n^i=\gamma_n z_n, \quad i=1,2,\ldots,N, \label{eq:38}\\
&z_{n+1}=\frac{1}{N}\sum_{i=1}^Nz_n^i, \quad n = 0, 1, \ldots, \label{eq:39}
\end{align}
where $A_{n,i}(x):=F_{n,i}(x)-f_{n,i},i=1,2,\ldots,N$.

\begin{theorem}\label{theo.convergence.perturb}
Assume that all the conditions of Theorem $\ref{theo.imp.convergence}$ and relations $\left(\ref{eq:36}\right)$, $\left(\ref{eq:37}\right)$ are fulfilled. If in addition $\frac{h_n+\delta_n}{\alpha_n} \to 0$ as $n\to +\infty$, then the sequence $\left\{z_n\right\}$ generated by $\left(\ref{eq:38}\right),\left(\ref{eq:39}\right)$ converges strongly to $\widehat{x}^*$ as $n\to +\infty$.
\end{theorem}
\begin{proof}
Let $\{x_n\}$ be the sequence of approximations defined by $(\ref{eq:12})$, $(\ref{eq:13})$. From $(\ref{eq:12})$ and $(\ref{eq:38})$, we have

$$ A_{n,i}(z_n^i)-A_{i}(x_n^i)+(\frac{\alpha_n}{N}+\gamma_n)\left(z_n^i-x_n^i\right)=\gamma_n (z_n-x_n),$$
or
\begin{align*}
\left(F_{n,i}(z_n^i)-F_{n,i}(x_n^i)\right) +\left(F_{n,i}(x_n^i)-F_{i}(x_n^i)\right) +\left(f_{n,i}-f_i\right)&+(\frac{\alpha_n}{N}+\gamma_n)\left(z_n^i-x_n^i\right) \\ 
&=\gamma_n (z_n-x_n).
\end{align*}
Therefore
\begin{align}
\langle F_{n,i}(z_n^i)&-F_{n,i}(x_n^i), J(z_n^i-x_n^i)\rangle +\left\langle F_{n,i}(x_n^i)-F_{i}(x_n^i),J(z_n^i-x_n^i)\right\rangle  \notag\\ 
& +\left\langle f_{n,i}-f_i,J(z_n^i-x_n^i)\right\rangle+(\frac{\alpha_n}{N}+\gamma_n)\left\|z_n^i-x_n^i\right\|^2 \notag\\ 
&=\gamma_n\left\langle z_n-x_n,J(z_n^i-x_n^i)\right\rangle \label{eq:40}
\end{align}
By relations $(\ref{eq:36}), (\ref{eq:37})$ and $(\ref{eq:40})$, we get
\begin{equation*}\label{eq:41}
(\frac{\alpha_n}{N}+\gamma_n)\left\|z_n^i-x_n^i\right\|^2 \le h_n g(\left \|x_n^i\right\|)\left\|z_n^i-x_n^i\right\|+\delta_n \left\|z_n^i-x_n^i\right\|+\gamma_n\left\|z_n-x_n\right\|\left\|z_n^i-x_n^i\right\|
\end{equation*}
Lemma $\ref{lemma.imp.boundedness}$ ensures the boundedness of the sequence $\left\{x_n^i\right\}$. Thus,  $\left\|x_n^i\right\| \le R$ for some $R>0$. Setting $\lambda_n=\left\|z_n-x_n \right\|$, from the last inequality we find
\begin{equation}\label{eq:42}
\left\|z_n^i-x_n^i\right\|\le \frac{N\gamma_n}{\alpha_n+N\gamma_n}\lambda_n+\frac{Ng(R)h_n}{\alpha_n+N\gamma_n}+\frac{N\delta_n}{\alpha_n+N\gamma_n}
\end{equation}
On account of $(\ref{eq:13}),(\ref{eq:39})$ and $(\ref{eq:42})$, we have
\begin{equation}\label{eq:43}
\lambda_{n+1}=\left\|z_{n+1}-x_{n+1}\right\|\le \frac{1}{N}\sum_{i=1}^N \left\|z_n^i-x_n^i\right\|\le \frac{N\gamma_n}{\alpha_n+N\gamma_n}\lambda_n+\frac{Ng(R)h_n}{\alpha_n+N\gamma_n}+\frac{N\delta_n}{\alpha_n+N\gamma_n}.
\end{equation}
Putting $p_n=\frac{\alpha_n}{\alpha_n+N\gamma_n}, b_n=\frac{Ng(R)h_n}{\alpha_n+N\gamma_n}+\frac{N\delta_n}{\alpha_n+N\gamma_n}$, from $(\ref{eq:43})$ we get $\lambda_{n+1}=(1-p_n)\lambda_{n}+b_n$. By virtue of the hypotheses of Theorem $\ref{theo.convergence.perturb}$, we have $\sum_{n=1}^\infty p_n =+\infty, \frac{b_n}{p_n}\to 0$ as $n\to +\infty$. Lemma $\ref{lem.sequence}$ implies that $\lambda_n=\left\|z_n-x_n \right\| \to 0$ as $n\to +\infty$. Finally, by Theorem $\ref{theo.imp.convergence}$, $x_n\to \widehat{x}^*\quad (n\to +\infty)$, hence $\left\|z_n-\widehat{x}^*\right\|\le \left\|z_n-x_n \right\| +\left\|x_n-\widehat{x}^* \right\|\to 0 (n\to +\infty)$. The proof of Theorem $\ref{theo.convergence.perturb}$ is complete.
\end{proof}

We now consider an explicit PIRM for solving system $(\ref{eq1})$, consisting of synchronous computation of intermediate approximations $z_{ni}$
\begin {equation}\label{eq:c21}
z_{ni}=z_n-\frac{1}{\gamma_n}\left\{A_i\left(z_n\right)+\frac{\alpha_n}{N}z_n\right\}=z_n-\tau_n\left\{A_i\left(z_n\right)+\frac{\alpha_n}{N}z_n\right\}, i=1,2,\ldots,N,
\end{equation}
and defining the next approximation $z_{n+1}$ as an average of intermediate approximations $z_{ni}$ 
\begin{equation}\label{eq:c22}
z_{n+1}=\frac{1}{N}\sum_{i=1}^N z_{ni},n=1,2,\ldots.
\end{equation}
\begin{lemma}\label{theo.exp.boundedness}
Suppose conditions B1-B2 are satisfied. Assume in addition the function $\frac{\varphi(s,t)}{t}$ is coercive in $t$ for every fixed $s>0.$
 Let $\left\{\alpha_n\right\}$ and $\left\{\gamma_n\right\}$ be positive sequences such that for all $n \geq 0$, $\alpha_n \le 1$, $ \gamma_n \ge 1,$ and
\begin{equation}\label{eq:c23}
\tau_n \le d, \quad \frac{\rho_X \left(\tau_n\right)}{\tau_n \alpha_n} \le d^2,
\end{equation}
where $\tau_n := 1/\gamma_n$ and $d \in (0,1)$  is a fixed number. Then starting from arbitrary $z_0 \in X,$ the sequence $\left\{z_n\right\}$ generalized by $(\ref{eq:c21})$, $\left(\ref{eq:c22}\right)$ is bounded.
\end{lemma}

\begin{proof}
A simple vertification shows that the sequence $\{z_n\}$ defined by  $(\ref{eq:c21})$ and $\left(\ref{eq:c22}\right)$ satisfies the relation
\begin{equation}\label{eq:c23a}
z_{n+1}=z_n - \frac{1}{N\gamma_n}\{A(z_n) + \alpha_n z_n\},
\end{equation}
where $A(z) := \sum_{i=1}^N A_i(z).$  By our assumptions, all the operators $A_i, i=1,\ldots,N,$ are continuous, $m$-accretive and $\varphi$-inverse uniformly accretive. Moreover, as it was shown in the proof of Theorem $\ref{theo.imp.convergence}$, $A_i$ is  bounded for every $i$, hence the operator $A: D(A) = X \to X$ is bounded, continuous and $m$-accretive. According to Lemma $\ref{lem.iquivalent}$, equation $A(z) = 0$  is equivalent to the consistent system $\left(\ref{eq1}\right)$. By [21, Theorem 5.1] the sequence $\{z_n\}$ defined by $(\ref{eq:c23a})$ is bounded.
\end{proof}

\begin{theorem}\label{theo.exp.convergence}
Assume that all the conditions of  Lemma $\ref{theo.exp.boundedness}$ are fulfilled. In addition, let $\alpha_n \to 0, \tau_n:= 1/\gamma_n \to 0$ as $n \to +\infty$, such that
\begin{equation}\label{eq:c24}
\sum_{i=1}^{\infty} \alpha_n \tau_n =+\infty,\quad \frac{\tau_n}{\alpha_n} \to 0 \quad \frac{\left|\alpha_n - \alpha_{n+1}\right|}{\tau_n \alpha_n^2}\to 0\quad \frac{\rho_X \left(\tau_n\right)}{\tau_n \alpha_n} \to 0.
\end{equation}Then the sequence $\left\{z_n\right\}$ generalized by $(\ref{eq:c21})$ and $\left(\ref{eq:c22}\right)$ converges to $\widehat{x}^*$ as $n\to \infty$.
\end{theorem}

\begin{proof}
Let $x^*_n$ be the unique solution of regularized equation $(\ref{eq:regu.sum.equ})$. It follows from Lemma $\ref{lem.regu.solution}$ that $\{x^*_n\}$ is bounded, hence there exists a constant $\tilde {d}>0$ such that 
$$ \left\|x^*_n - x^*_{n+1}\right\| \le \tilde{d}.$$
By Lemma $\ref{lem.ineqconvex}$, we have
\begin{align}\label{eq:c26}
\left\|z_{ni} - x^*_{n+1}\right\|^2 \le \left\|z_{ni} - x^*_{n}\right\|^2 \notag&+2\left\langle x^*_{n+1}- x^*_n ,J\left( x^*_{n} - z_{ni}\right) \right\rangle\\ 
&+2\left\langle x^*_{n+1}- x^*_n, J(x^*_{n+1}- z_{ni}) - J(x^*_{n} - z_{ni})\right\rangle.
\end{align}
Further, by Lemma $\ref{lem.geometry1}$, we get
\begin{align}\label{eq:c27}
\left\|z_{ni} - x^*_{n+1}\right\|^2 \le \left\|z_{ni} - x^*_{n}\right\|^2 \notag&+2\left\|z_{ni} - x^*_{n}\right\|. \left\|x^*_{n+1}- x^*_n\right\| \\ 
& + 16\left\|x^*_{n+1}- x^*_n\right\|^2 +c_1 \left(n\right) \rho_X \left(\left\|x^*_{n+1}- x^*_n\right\|\right),
\end{align}
where $c_1 (n) =8 \max\left\{2L,\left\|z_{ni} - x^*_{n+1}\right\| + \left\|z_{ni} - x^*_{n}\right\|\right\}$. Taking into account Lemma $\ref{theo.exp.boundedness}$ and the boundedness of the operators $A_i$ we conclude that the sequence $\left\{z_{ni}\right\}$ is also bounded, therefore there exist positive numbers $c_1, k_0$ such that $c_1 (n)\le c_1 $ and $\left\|z_{ni} - x^*_{n}\right\| \le k_0$ for all $n\ge 0$. Note that, if $H$ is a Hilbert space, then for all $0<\tau <\widehat{\tau}$ 
\begin{equation}\label{eq:c28}
\rho_X \left(\tau\right) \ge \rho_H \left(\tau\right) =\sqrt{1+\tau^2} -1 \ge \widehat{c}\tau^2,
\end{equation}
where $\widehat{c}=\left(\sqrt{1+\widehat{\tau}^2} +1\right)^{-1}$. \\
Now, summing up both sides of $(\ref{eq:c27})$ for $i=1,2 \ldots, N,$ and using Lemma $\ref{lem.geometry2}$, as well as inequality $(\ref{eq:c28})$ with $\widehat{\tau}:=\tilde{d} \ge \tau:=\left\|x^*_{n+1}- x^*_n\right\|$, we obtain
\begin{align*}
\sum_{i=1}^N\left\|z_{ni} - x^*_{n+1}\right\|^2 \le \sum_{i=1}^N\left\|z_{ni} - x^*_{n}\right\|^2\notag&+4Nk_0\frac{\left|\alpha_{n+1}-\alpha_n\right|}{\alpha_n}\left\|\widehat{x}^*\right\|\\
&+N\left(16\widehat{c}^{-1}+c_1\right) \rho_X \left(2\frac{\left|\alpha_{n+1}-\alpha_n\right|}{\alpha_n}\left\|\widehat{x}^*\right\|\right).
\end{align*}
\\
From the last inequality and the fact that $\rho_X \left(\tau\right) \le \tau$, one gets
\begin{equation}\label{eq:c30}
\sum_{i=1}^N\left\|z_{ni} - x^*_{n+1}\right\|^2 \le \sum_{i=1}^N\left\|z_{ni} - x^*_{n}\right\|^2 +c_3\frac{\left|\alpha_{n+1}-\alpha_n\right|}{\alpha_n},
\end{equation}
where $c_3=2N\left\|\widehat{x}^*\right\|\left(2k_0+16\widehat{c}^{-1}+c_1\right)$.\\
Next, we shall estimate the expression $\sum_{i=1}^N\left\|z_{ni} - x^*_{n}\right\|^2$. By Lemma $\ref{lem.ineqconvex}$ and $(\ref{eq:c21})$, we have
\begin{align}\label{eq:c31}
\left\|z_{ni} - x^*_{n}\right\|^2\notag &=\left\|z_{n} - x^*_{n} -\tau_n\left\{A_i\left(z_n\right)+\frac{\alpha_n}{N}z_n\right\}\right\|^2\\
\notag & \le \left\|z_{n} - x^*_{n}\right\|^2-2\tau_n \left\langle A_i\left(z_n\right)+\frac{\alpha_n}{N}z_n, J(z_{n} - x^*_{n})\right\rangle\\
&+2\left\langle z_{ni} - z_{n},J(z_{ni} - x^*_{n})-J(z_{n} - x^*_{n})\right\rangle.
\end{align}
Besides,
\begin{equation}\label{eq:c32}
\left\|z_{ni} - z_{n}\right\|=\tau_n\left\|A_i\left(z_n\right)+\frac{\alpha_n}{N}z_n\right\| \le M\tau_n,
\end{equation}
where $M=\sup \left\{\left\|A_i\left(z_n\right)+\frac{\alpha_n}{N}z_n\right\|: i=1,2,\ldots,N; n=1,2,\ldots\right\}$. Using $(\ref{eq:c32})$ and Lemma $\ref{lem.geometry1}$, one obtains
\begin{equation*}\label{eq:c33}
\left\langle z_{ni} - z_{n},J(z_{ni} - x^*_{n})-J(z_{n} - x^*_{n})\right\rangle \le 8M^2\tau_n^2 +c_2(n) \rho_X \left(M\tau_n\right),
\end{equation*}
where $c_2(n)=4 \max\left\{2L,\left\|z_{ni} - x^*_{n}\right\| + \left\|z_{n} - x^*_{n}\right\|\right\} \le c_2$ because of the boundedness of the sequences $\left\{z_{ni}\right\},\left\{z_{n}\right\}$ and $\left\{x^*_{n}\right\}$. Therefore
\begin{equation}\label{eq:c34}
\left\langle z_{ni} - z_{n},J(z_{ni} - x^*_{n})-J(z_{n} - x^*_{n})\right\rangle \le 8M^2\tau_n^2 +c_2\rho_X \left(M\tau_n\right).
\end{equation}
On the other hand, since the operators $A_i$ are accretive and $x^*_n$ is the solution of $(\ref{eq:regu.sum.equ})$, we have
\begin{align}\label{eq:c35}
\sum_{i=1}^N \langle A_i\left(z_n\right)+\frac{\alpha_n}{N}z_n,\notag& J(z_{n} - x^*_{n})\rangle =\sum_{i=1}^N \left\langle A_i\left(z_n\right)-A_i\left(x^*_n\right), J(z_{n} - x^*_{n})\right\rangle \\
\notag&+\left\langle\sum_{i=1}^N A_i\left(x^*_n\right) +\alpha_n x^*_n, J(z_{n} - x^*_{n})\right\rangle +\alpha_n\left\|z_{n} - x^*_{n}\right\|^2 \\
&\ge  \alpha_n\left\|z_{n} - x^*_{n}\right\|^2.
\end{align}
Now, summing the both sides of $(\ref{eq:c31})$ for $i=1,2,\ldots,N$ and taking account relations $(\ref{eq:c34}),(\ref{eq:c35})$, we get
\begin{equation}\label{eq:c36}
\sum_{i=1}^N\left\|z_{ni} - x^*_{n}\right\|^2 \le N\left\|z_{n} - x^*_{n}\right\|^2-2\tau_n \alpha_n\left\|z_{n} - x^*_{n}\right\|^2
+16NM^2\tau_n^2 +2Nc_2\rho_X \left(M\tau_n\right).
\end{equation}
Note that 
\begin{equation}\label{eq:c37}
\left\|z_{n+1} - x^*_{n+1}\right\|^2 \le \frac{1}{N^2}\left(\sum_{i=1}^N\left\|z_{ni} - x^*_{n+1}\right\|\right)^2 \le \frac{1}{N}\sum_{i=1}^N\left\|z_{ni} - x^*_{n+1}\right\|^2.
\end{equation}
From $(\ref{eq:c30}),(\ref{eq:c36}),(\ref{eq:c37})$, we get
\begin{align} \label{eq:c38}
\left\|z_{n+1} - x^*_{n+1}\right\|^2 \le \left\|z_{n} - x^*_{n}\right\|^2 \notag&-\frac{2\tau_n \alpha_n}{N}\left\|z_{n} - x^*_{n}\right\|^2\\
&+16M^2\tau_n^2 +2c_2\rho_X \left(M\tau_n\right)+c_3\frac{\left|\alpha_{n+1}-\alpha_n\right|}{\alpha_n}.
\end{align}
Setting $\lambda_n=\left\|z_{n} - x^*_{n}\right\|^2,  \quad  p_n=\frac{2\tau_n \alpha_n}{N}, \quad  b_n=16M^2\tau_n^2 +2c_2\rho_X \left(M\tau_n\right)+c_3\frac{\left|\alpha_{n+1}-\alpha_n\right|}{\alpha_n},$  we can rewrite $(\ref{eq:c38})$ as
\begin{equation*}\label{eq:c39}
\lambda_{n+1}\le (1-p_n)\lambda_n +b_n.
\end{equation*}
In the same manner as in the proof of Theorem 2.1, we can find positive integers $n_0$ and $m$, such that for all $n \geq n_0,  \rho_X(M \tau_n) \leq 5^m \rho_X(\tau_n).$   By Lemma $\ref{lem.sequence}$ and the hypothesis $(\ref{eq:c24})$, we conclude that $\lambda_n=\left\|z_{n} - x^*_{n}\right\|^2\to 0$ as $n\to +\infty$. Finally, by Lemma $\ref{lem.regu.solution}$,
\begin{equation*}\label{eq:c40}
\left\|z_{n} - \widehat{x}^*\right\| \le \left\|z_{n} - x^*_{n}\right\|+\left\|x^*_{n} - \widehat{x}^*\right\|\to 0,
\end{equation*}
which implies that $\left\{z_n\right\}$ converges to $\widehat{x}^*$.The proof of Theorem $\ref{theo.exp.convergence}$ is complete.
\end{proof}
\section{EQUATIONS  WITH SMOOTH ACCRETIVE OPERATORS}
\setcounter{lemma}{0}
\setcounter{theorem}{0}
\setcounter{equation}{0}
In this section, for solving system $(\ref{eq1})$ with smooth accretive operators $A_i$, we consider a parallel regularized Newton-type method (cf. \cite{AC2011, AD2013})
\begin{align}
&A_i(x_n)+\frac{\alpha_n}{N} (x_n -x^0_i)+(A'_i(x_n)+\frac{\alpha_n}{N} I)(x_{n}^i-x_n)=0 \label{eq:s26},i=1,2,\ldots,N,\\
&x_{n+1}=\frac{1}{N}\sum_{i=1}^N x_n^i, n=0,1,\ldots. \label{eq:s27}
\end{align}
The following assumptions will be needed throughout Section 3.
\begin{itemize}
\item [C1.]   System $(\ref{eq1})$ possesses an exact solution $\widehat{x}^*.$  The operators $A_i \quad (i = 1,\ldots, N)$ are accretive on a real Banach space $X$ and Fr$\acute{\rm e}$chet differentiable in a closed ball $B_r(\widehat{x}^*) \subset X$ with center $\widehat{x}^*$ and radius $r >0$. Moreover,
\begin{equation*}\label{eq:s25*}
\left\|A'_i(x)-A'_i(y)\right\| \le K\left\|x-y\right\|\,\, \forall x,y \in B_r(\widehat{x}^*),  i=1,\ldots,N. 
\end{equation*}
\item[C2.] The following componentwise source condition  (see \cite{AD2013,BK2006})   holds
\begin{equation*}\label{eq:s26*}
x_i^0-\widehat{x}^* =A'_i (\widehat{x}^*)v_i, 
\end{equation*}
where $x_i^0 \in B_r(\widehat{x}^*),\, v_i \in X, 1\le i\le N$.
\item[C3.]  The parameters $\alpha_n$ are chosen such that
\begin{equation*}\label{eq:s27*}
\alpha_n>0, \alpha_n \to 0, 1\le \frac{\alpha_n}{\alpha_{n+1}} \le \rho,
\end{equation*}
where the constant $\rho >1$.
\end{itemize}

\begin{theorem}\label{theo.exact.smooth}
Let all the assumptions ${\rm C1-C3}$ be satisfied. If  $\sum_{i=1}^N \left\|v_i\right\|$ is sufficiently small and $x_0$ is close enough to $\widehat{x}^*$, then there holds the estimate
\begin{equation}\label{eq:s28}
\left\|x_n -\widehat{x}^*\right\|=O(\alpha_n).
\end{equation}
\end{theorem}

\begin{proof} We suppose by induction that $x_n \in B_r(\widehat{x}^*)$ for some $n \geq 0$. Setting $e_n^i=x_n^i-\widehat{x}^*$ and $e_n=x_n-\widehat{x}^*$, from equation $(\ref{eq:s26})$ and assumption C2, we get
\begin{align*}
e_n^i&=e_n+x_n^i-x_n=e_n-\left(A'_i(x_n)+\frac{\alpha_n}{N} I\right)^{-1}\left(A_i(x_n)+\frac{\alpha_n}{N} (x_n -x^0_i)\right)\notag\\ 
&= \left(A'_i(x_n)+\frac{\alpha_n}{N} I\right)^{-1}\left[\left(A'_i(x_n)+\frac{\alpha_n}{N} I\right)e_n-\left(A_i(x_n)+\frac{\alpha_n}{N} (x_n -x^0_i)\right)\right]\notag\\
&=\left(A'_i(x_n)+\frac{\alpha_n}{N} I\right)^{-1}\left[\frac{\alpha_n}{N} (x^0_i-\widehat{x}^*)+A'_i(x_n)e_n-A_i(x_n)\right]\notag\\
&=\frac{\alpha_n}{N}\left(A'_i(x_n)+\frac{\alpha_n}{N} I\right)^{-1}A'_i (\widehat{x}^*)v_i+\left(A'_i(x_n)+\frac{\alpha_n}{N} I\right)^{-1}\left(A'_i(x_n)e_n-A_i(x_n)\right).\label{eq:s28}
\end{align*}
Using Lemma $\ref{lem.inequ.phu}$, from the last inequality we obtain
\begin{equation}\label{eq:s29}
\left\|e_n^i\right\|\le \frac{\alpha_n}{N}\left\|\left(A'_i(x_n)+\frac{\alpha_n}{N} I\right)^{-1}A'_i (\widehat{x}^*)\right\|\left\|v_i\right\|+\frac{N}{\alpha_n}\left\|A'_i(x_n)e_n-A_i(x_n)\right\|.
\end{equation}
Obviously, if  $x_t:=x_n+t(\widehat{x}^*-x_n),$ where $ 0\le t\le 1$, then $$\left\|x_t-\widehat{x}^*\right\|=(1-t)\left\|x_n-\widehat{x}^*)\right\|\le r,$$ hence $x_t \in B_r(\widehat{x}^*)$. \\
The second term of the right-hand side of $(\ref{eq:s29})$ can be estimated as
\begin{align}
\frac{N}{\alpha_n}\left\|A'_i(x_n)e_n-A_i(x_n)\right\|&= \frac{N}{\alpha_n}\left\|A_i(\widehat{x}^*)-A_i(x_n)+A'_i(x_n)(e_n)\right\|\notag\\ 
&=\frac{N}{\alpha_n}\left\|\int_0^1 \left(A'_i(x_n)-A'_i(x_t)\right)e_n dt \right\|\notag\\
&\le \frac{N}{\alpha_n}\int_0^1 Kt\left\|e_n\right\|^2 dt =\frac{KN}{2\alpha_n}\left\|e_n\right\|^2. \label{eq:s210}
\end{align}
On the other hand, we have
\begin{align*}
\left(A'_i(x_n)+\frac{\alpha_n}{N} I\right)^{-1}&-\left(A'_i(\widehat{x}^*)+\frac{\alpha_n}{N} I\right)^{-1}=\notag\\ 
&=\left(A'_i(x_n)+\frac{\alpha_n}{N} I\right)^{-1}\left[A'_i(\widehat{x}^*)-A'_i(x_n)\right]\left(A'_i(\widehat{x}^*)+\frac{\alpha_n}{N} I\right)^{-1}.
\end{align*}
Therefore, 
\begin{align*}
&\left\|\left(A'_i(x_n)+\frac{\alpha_n}{N} I\right)^{-1}A' (\widehat{x}^*)\right\| \le \left\|\left(A'_i(\widehat{x}^*)+\frac{\alpha_n}{N} I\right)^{-1}A'_i (\widehat{x}^*)\right\|\notag \\
&+\left\|\left(A'_i(x_n)+\frac{\alpha_n}{N} I\right)^{-1}\left[A'_i(\widehat{x}^*)-A'_i(x_n)\right]\left(A'_i(\widehat{x}^*)+\frac{\alpha_n}{N} I\right)^{-1}A'_i (\widehat{x}^*)\right\|.
\end{align*}
From the last inequality and using Lemma $\ref{lem.inequ.phu}$ as well as  assumption C1, we get
\begin{equation}\label{eq:s213}
\left\|\left(A'_i(x_n)+\frac{\alpha_n}{N} I\right)^{-1}A'_i(\widehat{x}^*)\right\| \le 2+\frac{2KN}{\alpha_n}\left\|e_n\right\|.
\end{equation}
Combining $(\ref{eq:s29}),(\ref{eq:s210}) $ and $(\ref{eq:s213})$, one has
\begin{equation}\label{eq:s214}
\left\|e_n^i\right\|\le \frac{2\alpha_n\left\|v_i\right\|}{N}+2K\left\|v_i\right\|\left\|e_n\right\|+\frac{NK}{2\alpha_n}\left\|e_n\right\|^2.
\end{equation}
By $(\ref{eq:s27})$ and $(\ref{eq:s214})$, we obtain
\begin{align}
\left\|e_{n+1}\right\|=\left\|x_{n+1}-\widehat{x}^*\right\|\le \frac{1}{N}\sum_{i=1}^N \left\|e_n^i\right\| &\le \frac{2\alpha_n\sum_{i=1}^N\left\|v_i\right\|}{N^2} \notag\\ 
&+\frac{2K\sum_{i=1}^N\left\|v_i\right\|}{N}\left\|e_n\right\|+\frac{NK}{2\alpha_n}\left\|e_n\right\|^2.\label{eq:s215}
\end{align}
 
Setting $\omega_n=\frac{N\left\|e_n\right\|}{\alpha_n}$  and using assumption  C3, from $(\ref{eq:s215})$ we find
\begin{align}
\omega_{n+1}&\le \frac{2\sum_{i=1}^N\left\|v_i\right\|}{N}\left(\frac{\alpha_n}{\alpha_{n+1}}\right)+\frac{2K\sum_{i=1}^N\left\|v_i\right\|}{N}\left(\frac{\alpha_n}{\alpha_{n+1}}\right)\omega_n+\frac{K}{2}\left(\frac{\alpha_n}{\alpha_{n+1}}\right)\omega_n^2\notag\\ 
& \le \frac{2\rho\sum_{i=1}^N\left\|v_i\right\|}{N}+\frac{2\rho K\sum_{i=1}^N\left\|v_i\right\|}{N}\omega_n+\frac{K\rho}{2}\omega_n^2\notag\\ 
&=a+b\omega_n+c\omega_n^2,\label{eq:s216}
\end{align}
where 
\begin{equation*}\label{eq:s217}
a=\frac{2\rho\sum_{i=1}^N\left\|v_i\right\|}{N}, ~ b=\frac{2\rho K\sum_{i=1}^N\left\|v_i\right\|}{N}, ~  c=\frac{K\rho}{2}.
\end{equation*}
If $\sum_{i=1}^N\left\|v_i\right\|$ is small enough, then $a,b$ will be small, hence
\begin{equation}\label{eq:s218}
b+2\sqrt{ac}<1, \frac{2a\alpha_0}{N}\le r\left(1-b+\sqrt{(1-b)^2-4ac}\right).
\end{equation}
Now if $x_0$ is sufficiently close to $\widehat{x}^*$ then
\begin{equation*}\label{eq:s219}
\omega_0=\frac{N\left\|e_0\right\|}{\alpha_0}=\frac{N\left\|x_0-\widehat{x}^*\right\|}{\alpha_0}\le M_{+}:=\frac{(1-b+\sqrt{(1-b)^2-4ac})}{2c}.
\end{equation*}
Lemma $\ref{lem.sequence2}$ applied to $(\ref{eq:s216})$ ensures that
\begin{equation*}\label{eq:s220}
\omega_n :=\frac{N\left\|e_n\right\|}{\alpha_n}\le l:=max\left\{\omega_0, M_{-}\right\}, \forall n \ge 0,
\end{equation*}
where $M_{-}=\frac{(1-b-\sqrt{(1-b)^2-4ac})}{2c}=\frac{2a}{(1-b+\sqrt{(1-b)^2-4ac})/2c}$.\\
In particular,
\begin{equation}\label{eq:s221}
\left\|x_{n+1}-\widehat{x}^*\right\|=\left\|e_{n+1}\right\|=\frac{\omega_{n+1}\alpha_{n+1}}{N}\le \frac{l\alpha_0}{N}.
\end{equation}
Observing that $\frac{\omega_{0}\alpha_{0}}{N}=\left\|x_{0}-\widehat{x}^*\right\|\le r$. From $(\ref{eq:s218})$, we have
\begin{equation*}\label{eq:s222}
\frac{M_{-}\alpha_0}{N}=\frac{2a\alpha_0}{N\left(1-b+\sqrt{(1-b)^2-4ac}\right)}\le r.
\end{equation*} 
Therefore, $\frac{l\alpha_0}{N}\le r$, hence $x_{n+1}\in B_r(\widehat{x}^*)$. Thus, the estimate $\omega_n \le l$ yields 
\begin{equation*}\label{eq:s223}
\left\|e_{n}\right\|=\frac{\omega_{n}\alpha_{n}}{N}\le\frac{l\alpha_{n}}{N}=O\left(\alpha_{n}\right).
\end{equation*}
The proof of Theorem $\ref{theo.exact.smooth}$ is complete.
\end{proof}
Now, assume that $A_i(x)=F_i(x)-f_i$ and instead of $\left(F_i,f_i\right)$, we only have approximations $\left(F^h_i,f^\delta_i\right)$, such that
\begin{equation}\label{eq:s224}
\left\|f^\delta_i - f_i\right\|\le \delta,
\end{equation}
and the operators $F^{h}_i,   (i = 1,\ldots, N),$ are accretive on a real Banach space $X$ and Fr$\acute{\rm e}$chet differentiable in a closed ball $B_r(\widehat{x}^*).$ Moreover, suppose that
\begin{equation}\label{eq:s230}
\left\|F^{h'}_i(x)-F^{h'}_i(y)\right\| \le K\left\|x-y\right\|\,\, \forall x,y \in B_r(\widehat{x}^*),i=1,2,\ldots,N.
\end{equation}
Also, assume that
\begin{align}
&\left\|F_i^h(\widehat{x}^*)-F_i(\widehat{x}^*)\right\|\le h,  \label{eq:s231}\\
&\left\|F_i^{h'}(\widehat{x}^*)-F^{'}_i(\widehat{x}^*)\right\|\le h.\label{eq:s232}
\end{align}
Given a starting point $x_0\in B_r(\widehat{x}^*)$, we define a sequence $\left\{x_n\right\}$:
\begin{align}
&\tilde{A}_i(x_n)+\frac{\alpha_n}{N} (x_n -x^0_i)+(\tilde{A}'_i(x_n)+\frac{\alpha_n}{N} I)(x_{n}^i-x_n)=0,i=1,2,\ldots,N \label{eq:s233}\\
&x_{n+1}=\frac{1}{N}\sum_{i=1}^N x_n^i,n=1,2,\ldots, \label{eq:s234}
\end{align}
where $\tilde{A}_i(x)=F_i^h(x)-f_i^{\delta}$. We define the stopping index $N(\delta,h)$ as
\begin{equation}\label{eq:s235}
N(\delta,h)=\max\left\{n:\alpha_n^2 \ge \frac{\delta+h}{\eta}\right\},
\end{equation}
where $\eta$ is a fixed parameter.\\
By the same argument as in the proof of Theorem $\ref{theo.exact.smooth}$, we come to the following estimate for iteration $(\ref{eq:s233})$ and $(\ref{eq:s234})$ 
\begin{equation}\label{eq:s236}
\left\|e_n^i\right\|\le \frac{2\alpha_n\left\|v_i\right\|}{N}+2\left\|v_i\right\|\left\|F_i^{'}(\widehat{x}^*)-F_i^{h'}(x_n)\right\|+\frac{N}{\alpha_n}\left\| F_i^{h'}(x_n)e_n-F_i^{h}(x_n)+f^\delta\right\|
\end{equation}
From conditions $(\ref{eq:s230})$ and $(\ref{eq:s232})$, we have
\begin{equation}\label{eq:s237}
\left\|F_i^{'}(\widehat{x}^*)-F_i^{h'}(x_n)\right\|\le \left\|F_i^{'}(\widehat{x}^*)-F_i^{h'}(\widehat{x}^*)\right\|+\left\|F_i^{h'}(\widehat{x}^*)-F_i^{h'}(x_n)\right\|\le h+K\left\|e_n\right\|.
\end{equation}
And using conditions $(\ref{eq:s224})$ and $(\ref{eq:s231})$, we have
\begin{align}
\left\| F_i^{h'}(x_n)e_n-F_i^{h}(x_n)+f^\delta\right\|&\le \left\| F_i^{h'}(x_n)e_n-F_i^{h}(x_n)+F_i^{h}(\widehat{x}^*)\right\|+\notag\\
&+\left\|F_i^{h}(\widehat{x}^*)-F_i(\widehat{x}^*)\right\|+\left\|f-f^\delta\right\|\le \frac{K}{2} \left\|e_n\right\|^2+h+\delta.\label{eq:s238}
\end{align}
Combining $(\ref{eq:s236}), (\ref{eq:s237}), (\ref{eq:s238})$ and noting that $\left\|e_{n+1}\right\|\le\frac{1}{N}\sum_{i=1}^N \left\|e_{n}^i\right\|$, we obtain
\begin{equation}\label{eq:s239}
\left\|e_{n+1}\right\|\le\left[\frac{2\sum_{i=1}^N\left\|v_i\right\|}{N}\left(h+\alpha_n\right)+\frac{h+\delta}{\alpha_n}\right]+\frac{2L\sum_{i=1}^N\left\|v_i\right\|}{N}\left\|e_{n}\right\|+\frac{K}{2\alpha_n}\left\|e_{n}\right\|^2.
\end{equation}
From $(\ref{eq:s235})$ we get $\frac{h}{\alpha_n}\le \alpha_n\eta \le \alpha_0\eta$. Setting $\omega_n=\frac{N\left\|e_n\right\|}{\alpha_n}$ we can rewrite $(\ref{eq:s239})$ as
\begin{align*}
\omega_{n+1} &\le\frac{\alpha_n}{\alpha_{n+1}}\left[\frac{2\sum_{i=1}^N\left\|v_i\right\|}{N}\left(\frac{h}{\alpha_{n}}+1\right)+\frac{h+\delta}{\alpha_n^2}+\frac{2L\sum_{i=1}^N\left\|v_i\right\|}{N}\omega_{n}+\frac{K}{2}\omega_{n}^2\right]\notag\\ 
& \le \rho\left[\left(\frac{2\sum_{i=1}^N\left\|v_i\right\|}{N}\left(\alpha_0\eta+1\right)+\eta\right)+\frac{2K\sum_{i=1}^N\left\|v_i\right\|}{N}\omega_{n}+\frac{K}{2}\omega_{n}^2\right]\notag\\ 
&=a+b\omega_n+c\omega_n^2,
\end{align*}
where 
\begin{equation*}\label{eq:s228}
a=\rho\left(\frac{2\sum_{i=1}^N\left\|v_i\right\|}{N}\left(\alpha_0\eta+1\right)+\eta\right),b=\frac{2\rho K\sum_{i=1}^N\left\|v_i\right\|}{N}, c=\frac{K\rho}{2}.
\end{equation*}
Again, if $\sum_{i=1}^N\left\|v_i\right\|$ and $\eta$ are small enough and $x_0$ is sufficiently close to $\widehat{x}^*,$ then arguing similarly as in the proof of Theorem $\ref{theo.exact.smooth}$, we can conclude that  $x_{n+1}\in B_r(\widehat{x}^*)$ and $\left\|x_{n}-\widehat{x}^*\right\|=O\left(\alpha_{n}\right)$ for $n=1,2,\ldots,N(\delta,h)$. Thus, we come to the following convergence result.

\begin{theorem}\label{theo.operperturb.convergence}
Assume that the assumptions {\rm C1-C3} hold for the exact operators $F_i,  i=1,\ldots,N, $  and conditions $(\ref{eq:s224})$, $(\ref{eq:s230}), (\ref{eq:s231})$ and $(\ref{eq:s232})$ are satisfied. If  $\sum_{i=1}^N \left\|v_i\right\|$ is sufficiently small and $x_0$ is close enough to $\widehat{x}^*$, then there holds the following estimate
\begin{equation}\label{eq:s229}
\left\|x_{n}-\widehat{x}^*\right\|=O\left(\alpha_{n}\right),\quad n=1,2,\ldots,N(\delta,h).
\end{equation}
\end{theorem}
Finally, taking into account the stopping rule (3.18), from Theorem $\ref{theo.operperturb.convergence}$ we obtain the convergence rate for the parallel regularized Newton-type method (3.16)-(3.17) in noisy data cases.
\begin{corollary}
Assume that all conditions of Theorem $\ref{theo.operperturb.convergence}$ are fulfilled. Then
\begin{equation}\label{eq:s229*}
\left\|x_{n^*}-\widehat{x}^*\right\|=O\left(\delta^{1/2}+h^{1/2}\right),
\end{equation}
where $n_* = N(\delta,h) +1.$
\end{corollary}
\section{CONCLUSION}
Most of existing solution methods for systems of ill posed operator equations deal with Hilbert spaces. In this paper we investigate 
two parallel iterative regularization methods and a parallel regularized Newton-type method for solving systems of equations involving 
$m$-accretive operators in Banach spaces. The convergence analysis of the proposed methods in both free-noise and noisy data cases is provided.
\section {ACKNOWLEDGMENTS}
The first author of this manuscript wishes to express his gratitude to Vietnam Institute for Advanced Study in Mathematics for  financial support.

\end{document}